\documentclass[12pt, a4paper]{article}

\usepackage{amsmath}
\usepackage{amsthm}
\usepackage{amssymb}
\usepackage{color}

\newtheorem{theorem}{Theorem}
\newtheorem{proposition}[theorem]{Proposition}
\newtheorem{lemma}[theorem]{Lemma}
\newtheorem{definition}[theorem]{Definition}

\newtheorem{corollary}[theorem]{Corollary}
\newtheorem{conjecture}[theorem]{Conjecture}

\newtheorem{remark}[theorem]{Remark}

\title{The running maximum of the Cox-Ingersoll-Ross process
with some properties of the Kummer function}
\author{Stefan Gerhold\thanks{Financial support from the Austrian Science Fund (FWF) under grant P~30750 is gratefully acknowledged. We thank
Veronika Pillwein for helpful comments.}, Friedrich Hubalek \\
TU Wien \\
1040 Vienna, Austria\\
\tt{sgerhold@fam.tuwien.ac.at}\\
\tt{fhubalek@fam.tuwien.ac.at}
 \and
 Richard B.~Paris\\
 Division of Computing and Mathematics,\\
 Abertay University, Dundee DD1 1HG, UK\\
 \tt{r.paris@abertay.ac.uk}
	}

\date{\today}

\numberwithin{equation}{section}
\numberwithin{theorem}{section}

\begin{document}

\maketitle

\begin{abstract}
We derive tail asymptotics for the running maximum of the Cox-Ingersoll-Ross process.
%Two auxiliary results are established. First, we obtain asymptotics for the Kummer function
%in a regime where the first parameter is proportional to the argument.
The main result is proved by the saddle point method, where the tail
estimate uses a new monotonicity
property of the Kummer function. This auxiliary result
is established by a computer algebra assisted proof.
Moreover, we analyse the coefficients of the
eigenfunction expansion of the running maximum distribution asymptotically.
\end{abstract}

MSC classes: 33C15, 41A60, 60G17

\section{Introduction}

The Cox-Ingersoll-Ross (CIR) process, also known as Feller diffusion,
is defined by the stochastic differential equation
\begin{equation}\label{eq:cir}
  dX_t = (\alpha-\beta X_t)dt + \sigma \sqrt{X_t}dW_t,
\end{equation}
where $W$ is a standard Brownian motion, and $\alpha,\beta,\sigma,X_0>0.$
This process has been intensively studied and is of particular interest
in mathematical finance, where its mean-reversion property, 
non-negativity and explicit transition density  make
it a popular choice for modelling stock volatility and other quantities~\cite{Gu12,ZhYaHo19}. The main results of this paper,
Corollaries~\ref{cor:main2} and~\ref{cor:main1}, give
asymptotics for  $\mathbb{P}[\max_{0\leq s\leq t}X_s \geq z]$
for fixed~$t$ and large~$z$. This is achieved by a saddle point
approximation of an integral representation involving
the Kummer function
\begin{equation}\label{eq:M series}
 M(a,b,x) = \sum_{n=0}^\infty \frac{(a)_n}{(b)_n\, n!}x^n,
\end{equation}
 where $(a)_n=a(a+1)\cdots (a+n-1)$ is the Pochhammer symbol.
This function satisfies the confluent hypergeometric ODE
\begin{equation}\label{eq:ode}
  x \frac{d^2M}{dx^2} + (b-x)\frac{dM}{dx}-aM =0.
\end{equation}
We refer to 13.7 and 13.8 in~\cite{DLMF,NI10} for many asymptotic
results about this function.  In the proof of our main theorems
we apply two auxiliary results
that may be of independent interest. In Appendix~\ref{se:asympt} we obtain asymptotics of the Kummer function $M(a,b,x)$, where $x\uparrow\infty$ and the parameter~$a$ is proportional
to~$x$. This result is known, but we give a new proof, again using the saddle point
method. Next, we give a computer algebra assisted proof
of the monotonicity of $|M(a,b,x)|$ with respect to $\mathrm{Im}(a)$ in
Appendix~\ref{se:mon}, which is also needed for one of our main results
(Corollary~\ref{cor:main2}). 
In Section~\ref{se:eigen} we
analyse the coefficients of the eigenfunction expansion of the
running maximum distribution. 
Appendix~\ref{se:zeros} contains a new proof of the known
fact that the $a$-zeros of $M(a,b,x)$ are negative and simple
for $b,x>0$.

\section{Tail asymptotics for the running maximum of the CIR process}\label{se:cir}

For any $\varepsilon\in(0,1)$ the scaled CIR process $\varepsilon X$ satisfies
\[
  d(\varepsilon X) = (\varepsilon \alpha-\beta\varepsilon X_t)dt + \sigma
  \sqrt{\varepsilon} \sqrt{\varepsilon X_t}dW_t.
\]
Since $\varepsilon \alpha<\alpha$, it follows from
a standard comparison result (Proposition~5.2.18 in~\cite{KaSh91}) that
 $\varepsilon X$ is almost surely dominated
by $Z^{(\varepsilon)}$ with dynamics 
\[
  dZ^{(\varepsilon)}_t = (\alpha-\beta Z^{(\varepsilon)}_t)dt + \sigma
     \sqrt{\varepsilon Z^{(\varepsilon)}_t}dW_t,
     \quad Z^{(\varepsilon)}_0=X_0.
\]
The family of processes $Z^{(\varepsilon)}$ converges to the
deterministic solution of $dZ^{(0)}_t = (\alpha-\beta Z^{(0)}_t)dt$ and satisfies a large deviations principle
for $\varepsilon\downarrow0$ (Theorem~1.2 in~\cite{BaCa11}). From the contraction
principle (Theorem 4.2.1 in~\cite{DeZe98}) applied to the functional $f\mapsto \max_{[0,t]}f$ it easily follows that
\begin{align}
  \mathbb{P}\Big[\max_{0\leq s\leq t}X_s \geq z\Big]
  &= \mathbb{P}\Big[\max_{0\leq s\leq t}z^{-1}X_s \geq 1\Big] \label{eq:P max} \\
   &\leq 
  \mathbb{P}\Big[\max_{0\leq s\leq t}Z_s^{(1/z)} \geq 1\Big]
  \leq \exp\big({-c}\sigma^{-2}z \big(1+o(1)\big)\big),
    \quad z\uparrow \infty, \label{eq:ldp bound}
\end{align}
where $c>0$ depends on $t,\alpha,\beta$.
This exponential bound was used recently in~\cite{GeGeGu20},
and prompted us to analyse the tail of the running maximum
of the CIR process in more detail, i.e., to determine the asymptotic
behavior of the left-hand side of~\eqref{eq:P max}.

Define the hitting time of level~$z$ by
\[
  \tau_{X_0\to z}:= \inf\{t\geq 0: X_t=z\}.
\]
It is a classical fact that the Laplace transform of a diffusion hitting
time can be expressed by the eigenfunctions of the infinitesimal
generator; see pp.~128--130 in~\cite{ItMc74}. For the CIR process,
these eigenfunctions are Kummer functions; we refer to~\cite{ChLi06} for details.
Corollary~4 of that paper states that
\[
  \mathbb{E}\big[e^{-s\tau_{X_0\to z}}\big] =
  \frac{M(s/\beta,2\alpha/\sigma^2,2\beta X_0/\sigma^2)}{M(s/\beta,2\alpha/\sigma^2,2\beta z/\sigma^2)}
\]
for $0<X_0<z$. 
By Laplace inversion, the law of the running maximum 
of the CIR process can be expressed as
\begin{align}
  \mathbb{P}\Big[\max_{0\leq s\leq t}X_s \geq z\Big] &=  \mathbb{P}[ \tau_{X_0\to z}\leq t] \label{eq:tail tau}\\
  &= \frac{1}{2\pi i}\int_{1-i\infty}^{1+i\infty} \frac{e^{t s}}{s}
  \frac{M(s/\beta,2\alpha/\sigma^2,2\beta X_0/\sigma^2)}{M(s/\beta,2\alpha/\sigma^2,2\beta z/\sigma^2)}ds. \label{eq:lap}
\end{align}
The main results of the present paper, namely Corollaries~\ref{cor:main2}
and~\ref{cor:main1} below, give asymptotics of this probability for
fixed~$t$ and large~$z$.
To simplify notation, we define
\begin{equation}\label{eq:def I}
  I(\lambda,b,x,y) := \frac{1}{2\pi i} \int_{1-i\infty}^{1+i\infty} \frac{e^{\lambda s}}{s}
  \frac{M(s,b,y)}{M(s,b,x)}ds
\end{equation}
for $0<y<x$ and $b,\lambda>0$, so that
\begin{equation}\label{eq:P I}
  \mathbb{P}\Big[\max_{0\leq s\leq t}X_s \geq z\Big] =
  I\Big(\beta t,\frac{2\alpha}{\sigma^2},\frac{2\beta z}{\sigma^2}, \frac{2\beta X_0}{\sigma^2}\Big).
\end{equation}
One of our main results (Theorem~\ref{thm:main2}
and its Corollary~\ref{cor:main2}) will be proven
conditionally, assuming the following statement.
\begin{conjecture}\label{conj}
  Let $b,u_0>0$ and $x>y>0.$ Then
  \begin{equation}\label{eq:quot}
    v\in\mathbb R_+\mapsto \bigg| \frac{M\big((u_0+iv)x,b,y\big)}{M\big((u_0+iv)x,b,x\big)} \bigg|
    \quad \text{decreases.}
  \end{equation}
\end{conjecture}
While we did not succeed in proving this conjecture, note
that a related inequality is established in Corollary~\ref{cor:mon},
namely that the denominator of~\eqref{eq:quot} increases with respect to $v>0$.
Define
\begin{align}
 \phi(u) &:= \lambda u - \psi\big( t_0(u) \big) \label{eq:def phi} \\
  &= \lambda u -\frac{1+\sqrt{1+4u}}{2}
     -u \log \Big( \frac{\sqrt{1+4u}+1}{\sqrt{1+4u}-1} \Big). \notag
\end{align}
For the definition of $\psi$ and $t_0=t_0(u)$, we refer
to~\eqref{eq:psi}, \eqref{eq:def psi} and~\eqref{eq:def t0} below.
It is easy to verify that
\begin{equation}\label{eq:phi u0}
  \phi'(u_0)=0 \quad \text{for} \quad
  u_0 := \big( 4 \sinh^2 (\tfrac12 \lambda) \big)^{-1},
\end{equation}
as well as
\begin{align*}
  \phi(u_0) &= -\tfrac12\big(1+\coth (\tfrac12 \lambda)\big), \\
  \phi''(u_0)&= u_0^{-1} \tanh(\tfrac12 \lambda),\\
  \sqrt{1+4u_0}&= \coth(\tfrac12 \lambda).
\end{align*}
In addition, we define
\begin{equation}\label{eq:def c1}
  C_1 := \frac{u_0^{b-3/2}}{\Gamma(b)\sqrt{ \phi''(u_0)}} \Big(\frac{\sqrt{1+4u_0}-1}{2}\Big)^{-b}
      (1+4u_0)^{1/4} 
\end{equation}
and
\begin{multline*}
  C_2 := \frac{u_0^{b/2-5/4}}{\sqrt{2\pi \phi''(u_0)}} \Big(\frac{\sqrt{1+4u_0}-1}{2}\Big)^{-b}
      (1+4u_0)^{1/4}e^{y/2} y^{-b/2+1/4} \\
      \times \exp\big(\tfrac12 \phi''(u_0)yu_0(1+4u_0)-y\sqrt{1+4u_0} \big).
\end{multline*}
Both of these quantities are constants, because~$C_2$ is used in a result where~$y>0$ is constant.
\begin{theorem}\label{thm:main1}
     Let $\lambda,b>0$, and let $y=y(x)>0$ be a function of~$x$
     satisfying $y(x) = o\big((x\log x)^{-1}\big)$ for $x\uparrow \infty$. Then
     the integral in~\eqref{eq:def I} satisfies
   \begin{equation*}%\label{eq:main}
       I(\lambda,b,x,y) \sim
  C_1   x^{b-1}\exp\big(x\phi(u_0)\big) 
   \end{equation*}
   as $x\uparrow\infty$.
\end{theorem}
The proof of Theorem~\ref{thm:main1} will be given towards the end
of this section.
It  uses the main results of the appendices
(Theorem~\ref{thm:kummer expans} and Corollary~\ref{cor:mon}). We first prove the following result, where the parameter~$y$
is constant.
Theorems~\ref{thm:main1} and~\ref{thm:main2} give first-order asymptotics.
As usual when applying the saddle point method, providing
further terms of the asymptotic expansion would be a matter of straightforward,
but cumbersome calculations.
\begin{theorem}\label{thm:main2}
  Suppose that Conjecture~\ref{conj} is true.
   Let $\lambda,b,y>0$. Then
   \begin{equation}\label{eq:main}
       I(\lambda,b,x,y) \sim
  C_2 x^{b/2-3/4}\exp\big(x\phi(u_0)+2\sqrt{yu_0x}\big) 
   \end{equation}
   as $x\uparrow\infty$.
\end{theorem}
\begin{proof}
  We rewrite the integral as
  \begin{align}
   \frac{1}{2\pi i}\int_{1-i\infty}^{1+i\infty}& \frac{e^{\lambda s}}{s}
  \frac{M(s,b,y)}{M(s,b,x)}ds
  =
    \frac{1}{2\pi i}\int_{\hat u-i\infty}^{\hat u+i\infty} \frac{e^{\lambda xu}}{u}
  \frac{M(ux,b,y)}{M(ux,b,x)}du  \notag \\
  &= \frac{1}{2\pi i}\bigg(
    \int_{\hat u-i x^{-2/5}}^{\hat u+ix^{-2/5}}
    + \int_{\substack{\mathrm{Re}(u)=\hat{u} \\ |\mathrm{Im}(u)|>x^{-2/5}}} \bigg)
    \frac{e^{\lambda xu}}{u}
  \frac{M(ux,b,y)}{M(ux,b,x)}du, \label{eq:integrals}
  \end{align}
  where $\hat{u}=\hat{u}(x)>0$ satisfies $\hat{u}(x)\uparrow u_0$
  and   will be fixed later. We will show in Lemmas~\ref{le:tail1}
  and~\ref{le:tail2}
  that the second integral in~\eqref{eq:integrals} is negligible, and focus now on the first integral. For~$u$ in
  its integration range and $x\uparrow \infty$,  the first term
  of the expansion (10.3.51)
  in~\cite{Te15} yields
  \begin{equation}\label{eq:num expans}
    M(ux,b,y) \sim \frac{\Gamma(b)}{\sqrt{2\pi}}\exp\big(\tfrac12y+2\sqrt{uxy}\big) (uxy)^{-b/2+1/4}.
  \end{equation}
  As for the denominator, Theorem~\ref{thm:kummer expans} implies
  \begin{equation}\label{eq:denom expans}
    M(ux,b,x) \sim \frac{\Gamma(b)}{\sqrt{2\pi}}(1+4u)^{-1/4} 
     \Big(\frac{\sqrt{1+4u}-1}{2}\Big)^b
     (ux)^{1/2-b} e^{x \psi(t_0)}.
   \end{equation}
  {}From these estimates, we obtain
  \begin{multline}\label{eq:cp}
    \frac{1}{2\pi i} \int_{\hat u-i x^{-2/5}}^{\hat u+ix^{-2/5}}
      \frac{e^{\lambda xu}}{u}
     \frac{M(ux,b,y)}{M(ux,b,x)}du \\
    \sim
      u_0^{b/2-5/4} \Big(\frac{\sqrt{1+4u_0}-1}{2}\Big)^{-b}
      (1+4u_0)^{1/4}e^{y/2} y^{-b/2+1/4} x^{b/2-1/4} \\
      \times \frac{1}{2\pi i} \int_{\hat u-i x^{-2/5}}^{\hat u+ix^{-2/5}}
      \exp\big( \lambda xu + 2\sqrt{uxy} -x \psi(t_0)\big)du.
  \end{multline}
  We put
  \[
    \chi(u,x) := x \phi(u) + 2\sqrt{uxy},
  \]
  so that $e^\chi$ is the integrand on the right-hand side of~\eqref{eq:cp}.
  We now define $\hat u(x)$ as the saddle point of this integrand, i.e.\ as the positive
  solution of
  \begin{equation}\label{eq:sp eq}
    0=\frac1x \frac{\partial}{\partial u} \chi(u,x) = 
    \lambda - \log\frac{\sqrt{1+4u}+1}{\sqrt{1+4u}-1} + \sqrt{\frac{y}{ux}}.
  \end{equation}
    It is easy to see that there is a unique solution for large~$x$, and that
    it converges to the (constant) saddle point~$u_0$ of $e^{x\phi(u)}$ as $x\uparrow\infty$.
  If we write the integration parameter as $u=\hat u+iv$, then the local
  expansion of~$\chi$ is
  \[
    \chi(u,x) = \chi(\hat{u},x) -\tfrac12 \chi''(\hat{u},x) v^2 + O(x^{-1/5}),
  \]
  where the derivative is with respect to~$u$, and the error term follows from
  $\chi'''(\hat{u},x)=O(x)$ and $v^3=O(x^{-6/5}).$ Now we can evaluate
  the integral in~\eqref{eq:cp} asymptotically:
  \begin{align}
     \frac{1}{2\pi i}\int_{\hat u-i x^{-2/5}}^{\hat u+ix^{-2/5}} &\exp\big(\chi(u,x)\big)du \notag \\
     &\sim
      \exp\big(\chi(\hat{u},x)\big)
     \frac{1}{2\pi }\int_{ -x^{-2/5}}^{x^{-2/5}} 
        \exp\big( {-\tfrac12} \chi''(\hat{u},x) v^2 \big) dv \notag \\
      &\sim
          \frac{\exp\big(\chi(\hat{u},x)\big)}{2\pi \sqrt{\chi''(\hat{u},x)}}
            \int_{ -\infty}^{\infty} e^{-z^2/2} dz 
     =  \frac{\exp\big(\chi(\hat{u},x)\big)}{\sqrt{2\pi \chi''(\hat{u},x)}}. \label{eq:sp main}
  \end{align}
  By inserting an ansatz $\hat{u}=u_0+w$ with $w=o(1)$ into~\eqref{eq:sp eq},
  it is easy to see that
  \[
    \hat{u} = u_0 - \sqrt{\frac{yu_0(1+4u_0)}{ x}}\big(1+O(x^{-1/2})\big).
  \]
  This implies
  \[
    \sqrt{\hat{u} x}= \sqrt{u_0 x}-\tfrac12\sqrt{y(1+4u_0)}+O(x^{-1/2})
  \]
  and (recall that~$u_0$ satisfies $\phi'(u_0)=0$)
  \[
    x\phi(\hat u) = x\phi(u_0) + \tfrac12 \phi''(u_0)yu_0(1+4u_0) +O(x^{-1/2}).
  \]
  We conclude
  \begin{align*}
      \chi(\hat u,x) = x \phi(u_0)+2\sqrt{u_0 xy} + \tfrac12 \phi''(u_0)yu_0(1+4u_0)
      -y\sqrt{1+4u_0}+O(x^{-1/2}).
  \end{align*}
  We insert this and $\chi''(\hat u,x)\sim x\phi''(u_0)$ into~\eqref{eq:sp main},
  and then use the resulting asymptotics in~\eqref{eq:cp}.
  Estimation of the second integral in~\eqref{eq:integrals} by Lemmas~\ref{le:tail1}
  and~\ref{le:tail2} below completes the proof.
  Clearly, it suffices to do the tail estimate for the upper half
  of the integration path.
\end{proof}
\begin{lemma}[Tail estimate for large~$\mathrm{Im}(u)$]\label{le:tail1}
  Let $\lambda,b,y>0.$ Then
  \begin{equation*}
     \bigg| \int_{\hat u+i \log x}^{\hat u+i\infty}
     \frac{e^{\lambda xu}}{u}  \frac{M(ux,b,y)}{M(ux,b,x)}du \bigg| 
     \leq
     \exp\Big( {-x \log x}\big(1+o(1)\big)\Big).
  \end{equation*}
\end{lemma}
\begin{proof}
Recall that the saddle point~$\hat u=\hat u(x)$ was defined in~\eqref{eq:sp eq}.
  By (10.3.51) in~\cite{Te15}, we have
  \[
    M(ux,b,x) = v^{-b/2} \exp\big(x\sqrt{2v}+O(x)\big), 
  \]
  where we write $u=\hat{u}+iv$ again. From~\eqref{eq:num expans}, we get
  \begin{align*}
     |M(ux,b,y)| &=\exp\big(2\mathrm{Re}\sqrt{uxy}+O(\log|ux|)\big) \\
     & \leq \exp\big(3\sqrt{xyv}\big)
  \end{align*}
  for large~$x$.
  We can thus estimate the integral by
  \begin{align*}
    e^{O(x)}\int_{\log x}^\infty v^{b/2}&\exp\big(
      {-x\sqrt{2v}}+3\sqrt{xyv}\big)dv 
      \leq  e^{O(x)}\int_{\log x}^\infty e^{-x\sqrt{v}}dv \\
      &= e^{O(x)}\int_{\log(x/2)}^\infty e^{-xz}zdz \\
      &\leq e^{O(x)}\int_{\log(x/2)}^\infty e^{-z(x-1)}dz \\
      &= \exp\Big( {-x}\log x\big(1+o(1)\big)\Big). \qedhere
  \end{align*}
\end{proof}
The final estimate for the proof of Theorem~\ref{thm:main2}
is provided by the following lemma. Note that
the exponential factor $\exp(-cx^{1/5})$ is negligible compared
to the power of~$x$ in~\eqref{eq:main}.
\begin{lemma}[Tail estimate for intermediate~$\mathrm{Im}(u)$]\label{le:tail2}
  Suppose that Conjecture~\ref{conj} is true.
  Let $\lambda,b,y>0$. Then there is a positive constant~$c$ such that
  \begin{equation*}
     \bigg| \int_{\hat u+i x^{-2/5}}^{\hat u+i\log x}
     \frac{e^{\lambda xu}}{u}  \frac{M(ux,b,y)}{M(ux,b,x)}du \bigg| 
     \leq
     \exp\Big( \chi(u_0,x) -c x^{1/5} + o(x^{1/5})\Big).
  \end{equation*}
\end{lemma}
\begin{proof}
  By assumption (Conjecture~\ref{conj}), $|M(ux,b,y)/M(ux,b,x)|$ is a decreasing
  function of $\mathrm{Im}(u)$. The integrand thus satisfies
  \begin{equation}\label{eq:frac est}
     \bigg|\frac{e^{\lambda xu}}{u}  \frac{M(ux,b,y)}{M(ux,b,x)} \bigg|
     \leq 
     e^{\lambda x\hat u+O(1)}
     \bigg|  \frac{M(ux,b,y)}{M(ux,b,x)} \bigg|_{u=\hat u+i x^{-2/5}}.
  \end{equation}
  We know from~\eqref{eq:denom expans}, the definition of~$\phi,$
  and $\phi'(u_0)=0$ that
  \begin{align}
     |M(ux,b,x)|\big|_{u=\hat u+i x^{-2/5}} &= 
     \exp\big(\lambda \hat u x
     -x\,\mathrm{Re}\, \phi(u+i x^{-2/5}) + O(\log x)\big) \notag \\
 &=\exp\big(\lambda \hat u x -x\phi(u_0)+\tfrac12 \phi''(u_0)x^{1/5}+O(1)\big).\label{eq:denom est}
  \end{align}
  By~\eqref{eq:num expans}, we have
  \begin{align}
     |M(ux,b,&y)|\big|_{u=\hat u+i x^{-2/5}} = 
     \exp\big(2 \mathrm{Re} \sqrt{uxy}+O(\log x)\big)\big|_{u=\hat u+i x^{-2/5}}\notag \\
 &=\exp\Big(2\sqrt{u_0 xy}\,\mathrm{Re}\sqrt{1+iu_0^{-1}x^{-2/5}+O(x^{-1/2})}+O(\log x)\Big) \notag \\
 &=\exp\big(2\sqrt{u_0 xy}+O(\log x)\big). \label{eq:num est}
  \end{align}
  Formulas \eqref{eq:frac est}--\eqref{eq:num est} imply
  \[
     \bigg|\frac{e^{\lambda xu}}{u}  \frac{M(ux,b,y)}{M(ux,b,x)} \bigg|
     \leq 
     \exp\big( \chi(u_0,x)-\tfrac12\phi''(u_0)x^{1/5} + O(x^{1/10})\big).
  \]
  The assertion is established, with $c=\tfrac12\phi''(u_0)>0,$
  by multiplying this estimate for the integrand
  with the length of the integration path.
\end{proof}

\begin{proof}[Proof of Theorem~\ref{thm:main1}]
   This proof is a simplified variant of the proof of
   Theorem~\ref{thm:main2}, which we have just completed.
   Instead of $\mathrm{Re}(u)=\hat u(x),$
   we integrate over the line $\mathrm{Re}(u)=u_0$, which does not depend on~$x$.
   Lemma~\ref{le:tail1} and its proof need no modification except
   replacing~$\hat u$ by~$u_0$.
   By (10.3.51) in~\cite{Te15}, we have
   \[
     M(ux,b,y) \sim \Gamma(b)(uxy)^{(1-b)/2}
     I_{b-1}(2\sqrt{uxy}), \quad x\uparrow\infty,
   \]
   where~$I_{\nu}$ is the modified Bessel function.
   For $0\leq\mathrm{Im}(u)\leq \log x$, this implies $M(ux,b,y) \sim 1,$
   since $I_{b-1}(z)\sim 2^{1-b}z^{b-1}/\Gamma(b)$ for $z\to0$. This can be used
   to adapt the proof of Lemma~\ref{le:tail2}. The estimate~\eqref{eq:frac est}
   becomes
   \[
   \bigg|\frac{e^{\lambda xu}}{u}  \frac{M(ux,b,y)}{M(ux,b,x)} \bigg|
     \leq     
       \frac{ e^{\lambda x u_0+O(1)}}{|M\big((u_0+i x^{-2/5})x,b,x\big)|},
   \]
   where we have applied the main result of Appendix~\ref{se:mon}
   (Corollary~\ref{cor:mon}). From this, it easily follows that this part of the tail
   satisfies
   \begin{equation*}
     \bigg| \int_{u_0+i x^{-2/5}}^{u_0+i\log x}
     \frac{e^{\lambda xu}}{u}  \frac{M(ux,b,y)}{M(ux,b,x)}du \bigg| 
     \leq
     \exp\Big( x \phi(u_0) -c x^{1/5} + o(x^{1/5})\Big).
  \end{equation*}
  It remains to approximate the central part of the integral. Using
  $M(ux,b,y) \sim 1$ again, we obtain
   \begin{multline*}
    \frac{1}{2\pi i} \int_{u_0-i x^{-2/5}}^{u_0+ix^{-2/5}}
      \frac{e^{\lambda xu}}{u}
     \frac{M(ux,b,y)}{M(ux,b,x)}du \\
    \sim
      \frac{\sqrt{2\pi}}{\Gamma(b)}u_0^{b-3/2} \Big(\frac{\sqrt{1+4u_0}-1}{2}\Big)^{-b}
      (1+4u_0)^{1/4}x^{b-1/2} \\
      \times \frac{1}{2\pi i} \int_{u_0-i x^{-2/5}}^{u_0+ix^{-2/5}}
      \exp\big( \lambda xu  -x \psi(t_0)\big)du.
  \end{multline*}
  The proof is now completed analogously to the proof of
  Theorem~\ref{thm:main2}. By~\eqref{eq:def phi} and~\eqref{eq:phi u0}, the exponent of the integrand has the expansion
  \begin{align*}
      \lambda xu  -x \psi\big(t_0(u)\big) &= x\phi(u) \\
      &= x\phi(u_0)-\tfrac12 \phi''(u_0)xv^2+O(x^{-1/5}),
  \end{align*}
  which implies
  \[
    \frac{1}{2\pi i} \int_{u_0-i x^{-2/5}}^{u_0+ix^{-2/5}}
      \exp\big( \lambda xu  -x \psi(t_0)\big)du
      \sim
      \frac{\exp\big(x \phi(u_0)\big)}{\sqrt{2\pi \phi''(u_0)x}}. \qedhere
  \]
\end{proof}
Now we return to the problem on CIR processes raised at the beginning of this section.
Define (see~\eqref{eq:def c1})
\[
  \hat{C}_1 := C_1\big|_{\lambda=\beta t,\, b=2\alpha/\sigma^2}
\]
and
\[
  \hat{C}_2 := C_2\big|_{\lambda=\beta t,\, b=2\alpha/\sigma^2,\, y=2\beta X_0/\sigma^2}.
\]
\begin{corollary}\label{cor:main2}
Let $\alpha,\beta,\sigma,t>0,$ and let $X_0=X_0(z)>0$ be a function
of $z$ that satisfies $X_0=o\big((z\log z)^{-1}\big)$ as $z\uparrow\infty$.
Then the CIR process defined in~\eqref{eq:cir}
  satisfies
 \begin{multline*}
    \mathbb{P}\Big[\max_{0\leq s\leq t}X_s \geq z\Big]
    \sim
    \hat{C}_1\Big(\frac{2\beta z}{\sigma^2}\Big)^{2\alpha/\sigma^2-1}
    \exp\Big(
      {-\frac{\beta}{\sigma^2}\big(1+\coth(\tfrac12 \beta t)\big)z}
        \Big), \quad z\uparrow \infty.
  \end{multline*}
\end{corollary}
\begin{proof}
  Immediate from~\eqref{eq:P I} and
  Theorem~\ref{thm:main1}.
\end{proof}
Analogously, using Theorem~\ref{thm:main2}, we get the following result.
\begin{corollary}\label{cor:main1}
  Suppose that Conjecture~\ref{conj} is true. 
  Let $\alpha,\beta,\sigma,X_0,t>0.$ Then the CIR process
  satisfies
  \begin{multline*}
    \mathbb{P}\Big[\max_{0\leq s\leq t}X_s \geq z\Big]
    \sim
    \hat{C}_2\Big(\frac{2\beta}{\sigma^2}\Big)^{\alpha/\sigma^2-3/4}
    z^{\alpha/\sigma^2-3/4} \\
    \times\exp\bigg(
      {-\frac{\beta}{\sigma^2}\big(1+\coth(\tfrac12 \beta t)\big)z}
        + \frac{2\beta \sqrt{X_0}}{\sigma^2 \sinh(\tfrac12 \beta t)} \sqrt{z}
        \bigg), \quad z\uparrow \infty.
  \end{multline*}
\end{corollary}
Note that the cruder LDP bound~\eqref{eq:ldp bound} correctly
captures the dependence of the exponential factor
of the tail asymptotics on~$z$ and $\sigma$. As a consistency check, we compare our
results with the tail of the CIR marginal distribution. From the well-known explicit transition
density (see~(4) in~\cite{ChLi06}), we obtain, for fixed $t>0$,
\begin{align*}
  \mathbb{P}[X_t\geq z] &=e^{O(\log z)}\int_z^\infty
  \exp\bigg({-\frac{2\beta e^{\beta t}y}{\sigma^2(e^{\beta t}-1)}}\bigg)
  I_{2\alpha/\sigma^2-1}\bigg(
  \frac{4\beta \sqrt{X_0e^{\beta t}y}}{\sigma^2(e^{\beta t}-1)}
  \bigg)dy \\
  &= e^{O(\sqrt{z})}\int_z^\infty
  \exp\bigg({-\frac{2\beta e^{\beta t}y}{\sigma^2(e^{\beta t}-1)}}\bigg)dy\\
  &=\exp\bigg({-\frac{2\beta e^{\beta t}}{\sigma^2(e^{\beta t}-1)}}z+ O(\sqrt{z})\bigg) \\
  &=\exp\bigg(
      {-\frac{\beta}{\sigma^2}\big(1+\coth(\tfrac12 \beta t)\big)z}+ O(\sqrt{z})\bigg),
  \quad z\uparrow\infty.
\end{align*}
Therefore, logarithmic tail asymptotics of the marginal and the running maximum
agree. This is not surprising, because for paths having a very large running maximum $\max_{0\leq s\leq t}X_s,$ this maximum is typically realized
close to time~$t$, where the process has had the most time to deviate
from its initial value.

While there is a considerable literature on tail asymptotics
of diffusion hitting times (also known as first-passage times),  asymptotics with respect to the level
have received less attention. By~\eqref{eq:tail tau}, level asymptotics
are equivalent to tail asymptotics of the running maximum.
Some results related to ours are given in~\cite{Mc09,NoRiSa85}.
However, Assumption~(1.3) of~\cite{Mc09} is not satisfied for the CIR process.
In~\cite{NoRiSa85}, Corollary~1 is of interest for our work. It gives
level asymptotics for the density of the hitting time, for a rather general
diffusion that has an invariant distribution, which is the case for the CIR process.
By integrating this density approximation, we can formally get asymptotics
for the cumulative distribution function, which translates into running maximum
tail asymptotics by~\eqref{eq:tail tau}. 
However, the result of this heuristic argument does \emph{not} agree with our findings.
While this may simply be
a case where integration and asymptotics (with respect to a parameter) do not commute,
we note that several steps in~\cite{NoRiSa85} appear to be non-rigorous.
For instance, no argument is given for the interchange of limit and summation in the proof of
Corollary~1.

\section{Eigenfunction expansion}\label{se:eigen}

The integrand in~\eqref{eq:def I} has infinitely many poles, all of which
are simple and in $(-\infty,0]$ (see Proposition~\ref{prop:zeros} for a new proof of the latter two properties).
We denote them by
\[
  0 > -s_0 > -s_1 > \cdots.
\]
Using the residue theorem and some asymptotic properties of the Kummer function and its $a$-zeros,
it is not hard to show that
\begin{align}
  I(\lambda,b,x,y)&=
  1+ \sum_{k=0}^\infty \mathrm{res}_{s=-s_k} \frac{e^{\lambda s}}{s}\frac{M(s,b,y)}{M(s,b,x)}\notag \\
  &=1-\sum_{k=0}^\infty \frac{e^{-\lambda s_k}}{s_k}
    \frac{M(-s_k,b,y)}{M'(-s_k,b,x)}. \label{eq:I series}
\end{align}
Throughout this section, $M'$ denotes the derivative with respect to the first parameter.
By~\eqref{eq:tail tau} and~\eqref{eq:P I}, this gives an expansion of the cumulative
distribution function of the CIR hitting time, which is well known.
We refer to Propositions~1 and~2 in~\cite{Li04}, and to~\cite{Ke80} for a classical reference
on such expansions for general diffusions. Proposition~2 in~\cite{Li04} also
gives the asymptotic behavior of the expansion coefficients for large~$k$. In the spirit
of the above results, we analyse the coefficients as $x\uparrow\infty$.

%For $R>0$, write $C_R$ for the contour consisting of the half-circle connecting the points $iR$
  %and $-iR$ through the left half-plane, augmented by two horizontal lines of length one to
  %the right of each of its endpoints.
  %Let~$L$ be the vertical line $\{1+i t : t\in\mathbb R\}$, and
  %\[
   % L_{\geq R} = \{1+i t : t\leq -R\} \cup \{1+i t : t\geq R\}.
  %\]
 % We now let the half-circle pass midway between the $s$-zeros of $M(s,b,x)$, by defining
  %\[
   % R_n := \frac{s_n+s_{n+1}}{2}.
  %\]
 % By the residue theorem,
  %\begin{multline}\label{eq:residues}
   %I(\lambda,b,x,y) 
   % = 1+ \sum_{k=0}^n \mathrm{res}_{s=-s_k} \frac{e^{\lambda s}}{s}\frac{M(s,b,y)}{M(s,b,x)} \\
   % \quad     + \frac{1}{2\pi i}\Big(\int_{L_{\geq {R_n}}} - \int_{C_{R_n}} \Big)\frac{e^{\lambda s}}{s}\frac{M(s,b,y)}{M(s,b,x)}ds.
  %\end{multline}
  %It is not hard to show that the last two integrals converge to zero as $n\to\infty$,
 % using (10.3.57) in~\cite{Te15} and 10.17.3 in~\cite{DLMF,NI10}.
 % These expansions can also be used to prove that the sum in~\eqref{eq:residues}
  %converges for $n\to\infty$, and therefore

%\begin{equation}\label{eq:I sum}
  %I(\lambda,b,x,y) = 1-\sum_{k=0}^\infty \frac{e^{-\lambda s_k}}{s_k}
    %\frac{M(-s_k,b,y)}{M'(-s_k,b,x)}.
%\end{equation}
%\inlinecomment{clarify that $'$ is for first parameter}

%\inlinecomment{Insert references to density expansions for Bessel process hitting times (special case of our series)}

{}From 13.2.39 in~\cite{DLMF,NI10} and the expansion for~$M$ in Theorem~1
of~\cite{Pa13}, it follows that
\begin{align}
  M(s,b,x)
  &=e^x M(b-s,b,-x) \notag \\
  &=\frac{e^x \Gamma(b)}{\Gamma(b-s)}\Big(
    \frac{x^{s-b}\Gamma(b-s)}{\Gamma(s)}
    + x^{-s}e^{-x} \cos \pi s + O\big(x^{s-b-1}/\Gamma(s)\big) \Big). \label{eq:as eq}
\end{align}
From this we easily see that $s_k=s_k(x)$ converges to~$k$
for $x\uparrow\infty$. (In contrast to that, for fixed~$x$ and large~$k$
 the behavior of $s_k\sim \pi^2 k^2/(4x)$ is quadratic, by 13.9.10 in~\cite{DLMF,NI10}.)
 The asymptotics of $s_k$  for large~$x$ can be found  by setting the leading term
of~\eqref{eq:as eq} to zero, namely
\[
x^{s-b} \Gamma(b-s) / \Gamma(s) +x^{-s}e^{-x} \cos \pi s = 0.
 \]
 Since 
\begin{align*}
   \frac{\Gamma(b-s)}{\Gamma(s)} &= \frac{(s)_{k+1}\Gamma(b-s)}{\Gamma(s+k+1)} \\
  &\sim \Gamma(b+k)(s+k)(-1)^k (-s)(-s-1)\dots(-s-k+1)\\
   &\sim \Gamma(b+k)(s+k)(-1)^k k!,
\end{align*}
we have
\[
  s+k \sim -\frac{x^{b+2k}e^{-x}}{k! \, \Gamma(b+k)}.
\]
We have proved:
\begin{lemma}
  For $k\in\mathbb{N}_0$, we have
  \begin{equation}\label{eq:asympt sk}
    s_k = k + \frac{x^{b+2k}e^{-x}}{k!\, \Gamma(b+k)}\big(1+o(1)\big), \quad x\uparrow \infty.
  \end{equation}
\end{lemma}

\begin{definition}
  The harmonic numbers of order~$\nu$ are defined by
  \[
    H^{(\nu)}_n := \sum_{k=1}^n \frac{1}{k^{\nu}},
  \]
  and $H_n:=H^{(1)}_n$.
\end{definition}

\begin{theorem}
  Fix $k\geq0$.
  For $x\uparrow \infty$, we have
   \begin{align*}
    M'(-s_k,b,x) &= (-1)^k k! \sum_{r=k+1}^\infty \frac{x^r(r-k-1)!}{(b)_r\, r!} +O(x^k) \\
    &= \frac{(-1)^k  x^{k+1}}{(b)_{k+1}(k+1)}
    \, {}_2F_2\Big( \genfrac{}{}{0pt}{}{1,1}{b+k+1,k+2} \Big | x \Big)  +O(x^k).
  \end{align*}
  For $k=0$, we have, more precisely:
  \[
    M'(-s_0,b,x) = \frac{x}{b}\, {}_2F_2\Big( \genfrac{}{}{0pt}{}{1,1}{b+1,2} \Big | x \Big)
    -2s_0 \sum_{r=1}^\infty \frac{H_{r-1}\, x^r}{(b)_r\, r} + O(x^b e^{-x}).   
  \]
\end{theorem}
\begin{proof}
  By~\eqref{eq:M series},
  \begin{equation}\label{eq:M'}
    M'(s,b,x) = \sum_{r=1}^\infty \frac{(s)_r\, x^r}{(b)_r\, r!}(\psi(s+r)-\psi(r))
    = \sum_{r=1}^\infty \frac{(s)_r\, x^r}{(b)_r\, r!} \sum_{m=1}^r\frac{1}{m-1+s},
  \end{equation}
  where~$\psi=\Gamma'/\Gamma$ denotes the digamma function.
  This formula, as well as many others concerning the derivatives of~$M$ with respect to its parameters,
  also appears in~\cite{AnGa08}.  
  If we put $s=-s_k$ in~\eqref{eq:M'}, then
  the sum $\sum_{r=1}^k$ is zero for $k=0$ and $O(x^k)$ for $k\geq1$, and can thus
  be ignored in the following. For $r>k$, we write the Pochhammer symbol as
  \begin{equation}\label{eq:poch split}
    (s)_r =  (s)_k\, (s+k) (s+k+1)_{r-k-1}.
  \end{equation}
  We may assume $|s+k|<1$, as we intend to put $s=-s_k\to -k.$ Then,
  the last factor is
  \begin{align*}
    (s+k+1)_{r-k-1} &= \prod_{j=0}^{r-k-2}(j+1)\Big(1+\frac{s+k}{j+1}\Big) \\
    &= (r-k-1)!\, \exp\Big( \sum_{j=0}^{r-k-2} \log\Big(1+\frac{s+k}{j+1}\Big) \Big)\\
    &= (r-k-1)!\, \exp\Big({-\sum_{j=0}^{r-k-2}} \sum_{\nu=1}^\infty
      \frac{1}{\nu}\Big(\frac{-s-k}{j+1}\Big)^\nu \Big).
  \end{align*}
  Since
  \begin{align*}
    \sum_{j=0}^{r-k-2} \sum_{\nu=2}^\infty \frac{1}{\nu}\Big(\frac{s+k}{j+1}\Big)^\nu 
    = \sum_{\nu=2}^\infty \frac{H^{(\nu)}_{r-k-1}}{\nu}(s+k)^\nu
    = O((s+k)^2)
  \end{align*}
  as $s+k\to0$ (recall that $k$ is fixed throughout), uniformly with respect to~$r$, we obtain
  \begin{align*}
    (s+k+1)_{r-k-1} &= (r-k-1)!\ \exp\Big({-\sum_{j=0}^{r-k-2}}
      \frac{s+k}{j+1} + O((s+k)^2) \Big) \\
    &=(r-k-1)!\, \big(1+H_{r-k-1}(s+k) + O((s+k)^2) \big).
   \end{align*}
   We proceed with the first factor in~\eqref{eq:poch split}:
   \begin{align*}
     (s)_k &= (-1)^k k!\, \exp\Big( \sum_{j=0}^{k-1} \log\Big(
        1+ \frac{-s-k}{k-j}\Big) \Big) \\
     &= (-1)^k k!\, \exp\Big({- \sum_{j=0}^{k-1}} \sum_{\nu=1}^\infty
       \frac{1}{\nu}\Big( \frac{s+k}{k-j} \Big)^\nu \Big) \\
     &= (-1)^k k!\, \exp\Big({- \sum_{j=0}^{k-1}} \frac{-s-k}{k-j} + O((s+k)^2) \Big) \\
     &=(-1)^k k!\, \big(1-H_k(s+k) + O((s+k)^2)  \big).
   \end{align*}
   It is easy to see that the last sum in~\eqref{eq:M'} satisfies
   \begin{equation*}%\label{eq:sum m}
     \sum_{m=1}^r\frac{1}{m-1+s} = \frac{1}{s+k} + H_{r-k-1} - H_k + O((s+k)),
   \end{equation*}
   as $s\to-k$, uniformly with respect to $r>k$. Using this and the estimate we  found
   for $(s)_r$  yields
   \begin{align}
     &\sum_{r=k+1}^\infty \frac{(s)_r\, x^r}{(b)_r\, r!}  \sum_{m=1}^r\frac{1}{m-1+s} \notag \\
      &= (-1)^k k! \sum_{r=k+1}^\infty \frac{x^r(r-k-1)!}{(b)_r\, r!}
      \Big( 1-2(s+k)(H_k-H_{r-k-1} ) + O((s+k)^2) \Big). \notag
   \end{align}
   Since
   \begin{align*}
     \sum_{r=k+1}^\infty \frac{x^r(r-k-1)!}{(b)_r\, r!} &= 
     \frac{x^{k+1}}{(b)_{k+1}(k+1)!}\, {}_2F_2\Big(\genfrac{}{}{0pt}{}{1,1}{b+k+1,k+2}\Big|x \Big) \\
     &\sim \frac{x^{k+1}}{(b)_{k+1}(k+1)!}
       \frac{\Gamma(b+k+1)(k+1)!\, e^x}{x^{2k+b+1}} \\
     &= O(x^{-k-b}e^x), \quad x\uparrow \infty,
   \end{align*}
   where we have used 16.11.7  in \cite{DLMF,NI10},
   it follows that
   \begin{align}
     \sum_{r=k+1}^\infty \frac{(s)_r\, x^r}{(b)_r\, r!} & \sum_{m=1}^r\frac{1}{m-1+s} 
     =    \frac{(-1)^k x^{k+1}}{(b)_{k+1}(k+1)}\,
        {}_2F_2\Big(\genfrac{}{}{0pt}{}{1,1}{b+k+1,k+2}\Big|x \Big) \label{eq:sum k+1}  \\
     &-2  (-1)^k k!(s+k) \sum_{r=k+1}^\infty \frac{x^r(r-k-1)!}{(b)_r\, r!}(H_k-H_{r-k-1}) \notag \\
     & + O\big( (s+k)^2  x^{-k-b}e^x \big). \label{eq:sum k+1 2} \notag
   \end{align}
   As mentioned at the beginning of the proof, it suffices to estimate the sum in~\eqref{eq:sum k+1},
   with $s=-s_k$.
   For $k=0$, the result now follows from~\eqref{eq:asympt sk}. For $k\geq 1$, the claim follows from~\eqref{eq:asympt sk},  the fact that $H_k-H_{r-k-1}=O(r)$
   and the expansion of~${}_2F_2$ (see 16.11.7  in \cite{DLMF,NI10}).
\end{proof}

\begin{corollary}
  The asymptotic behavior of the summands in~\eqref{eq:I series} for $x\uparrow\infty$ is
  \begin{equation}\label{eq:as summand}
    \frac{e^{-\lambda s_k}}{s_k}
    \frac{M(-s_k,b,y)}{M'(-s_k,b,x)} =
    \begin{cases}
       1+o(1), & k = 0, \\
       \frac{(-1)^k M(-k,b,y)}{e^{\lambda k}k\cdot k!\, \Gamma(b)}
         x^{k+b}e^{-x}
         +o(x^{k+b}e^{-x}),  & k\geq1.
    \end{cases}
  \end{equation}
\end{corollary}
%We did not write the statement as an asymptotic equality, because
%we cannot rule out the possibility that $M(-k,b,y)$ could vanish.
By~\eqref{eq:as summand}, the summand $k=0$ almost cancels with~$1$ on the right-hand side of~\eqref{eq:I series}.
With a bit of extra work, it can be shown that the net contribution
of these two summands satisfies
\[
1- \frac{e^{-\lambda s_0}}{s_0}
    \frac{M(-s_0,b,y)}{M'(-s_0,b,x)} \sim s_0\Big( \lambda +\frac{y}{b} \,
     {}_2F_2\Big(\genfrac{}{}{0pt}{}{1,1}{b+1,2}\Big|y \Big)\Big),\quad x\uparrow \infty.
\]

\appendix

\section{Asymptotics of $M(a,b,x)$ for $a\approx x$}\label{se:asympt}

We use the saddle point method to analyse the Kummer function $M(a,b,x)$ 
for $x\uparrow\infty$, with~$b$ fixed and~$a$ of the same growth
order as~$x$. It is important to note that this result is not new,
as it can be obtained from the expansion (27.4.64) in~\cite{Te15}
by putting $a=ux,c=b$ and $z=1/u$. There, a different method was used,
and $z$ is assumed to be real and positive, but the latter constraint
can be easily relaxed to $\mathrm{Re}(z)>0,$ by inspection
of the proof in~\cite{Te15}.
 \begin{theorem}\label{thm:kummer expans}
   Let $b\in\mathbb{C}\setminus\{0,-1,\dots\}$ and $\mathrm{Re}(u)>0.$
   Then
   \[
     M(ux,b,x) \sim \frac{\Gamma(b)}{\sqrt{2\pi}(1+4u)^{1/4}}
     \Big(\frac{\sqrt{1+4u}-1}{2}\Big)^b
     (ux)^{1/2-b} e^{x \psi(t_0)}
   \]
   as $x\uparrow\infty$, where
   \begin{equation}\label{eq:psi}
     \psi(t_0)= \frac{1+\sqrt{1+4u}}{2}
     +u \log \Big( \frac{\sqrt{1+4u}+1}{\sqrt{1+4u}-1} \Big).
   \end{equation}
   This holds uniformly with respect to~$u$ if~$u$ is bounded and bounded
   away from zero, and $|\arg u|\leq\tfrac12 \pi-\varepsilon$
   for some $\varepsilon>0$.
 \end{theorem}
\begin{proof}
  By (10.1.6) in~\cite{Te15}, we have
  \begin{equation}\label{eq:Te repr}
    M(a,b,x) = \frac{\Gamma(1+a-b)\Gamma(b)}{\Gamma(a)}
    \frac{1}{2\pi i}\int_0^{(1+)} e^{xt}
    t^{a-1}(t-1)^{b-a-1}dt,
  \end{equation}
  where $a=ux$.
  The integration path starts and ends at zero
  and goes around $t=1$ counterclockwise. Defining
  \begin{equation}\label{eq:def psi}
    \psi(t) := t+ u \log \frac{t}{t-1} 
    \quad \text{and} \quad
    f(t) := (t-1)^{b-1}/t,
  \end{equation}
  we can write the integral as
  \[
     \frac{1}{2\pi i}\int_0^{(1+)} e^{x\psi(t)}f(t)dt.
  \]
  Equating the first derivative
  $    \psi'(t) = 1-\frac{u}{t(t-1)}$
  to zero, we find a saddle point at
  \begin{equation}\label{eq:def t0}
    t_0=t_0(u) := \frac{1+\sqrt{1+4u}}{2}.
  \end{equation}
  The second derivative of~$\psi$ at the saddle point is
  \[
    \psi''(t_0) = \frac{u(2t-1)}{t^2(t-1)^2}\bigg|_{t=t_0}
    =\frac{\sqrt{1+4u}}{u}  =: | \psi''(t_0)| e^{i\theta}. 
  \]
  The integration contour is deformed in order to pass through~$t_0$.
  If $u\in(0,\infty),$ then $t_0>1$ is real, and the contour
  is vertical at~$t_0$. For general~$u$, we let the contour be such that
  $\arg(t-t_0)=\tfrac12\pi-\tfrac12\theta$ holds for $|t-t_0|$ small
  after~$t$ traverses the saddle point. Now we can apply
  Theorem 4.7.1 in~\cite{Ol74}. It is straightforward to see
  that the contour can be chosen such that the inequality
  before that theorem is satisfied.
   Its other assumptions are clearly
  satisfied as well, and we obtain
  \begin{equation}\label{eq:int olver}
     \frac{1}{2\pi i}\int_0^{(1+)} e^{x\psi(t)}f(t)dt
     \sim 
    (1+4u)^{-1/4}
     \Big(\frac{\sqrt{1+4u}-1}{2}\Big)^b
     (2\pi ux)^{-1/2} e^{x \psi(t_0)}.
  \end{equation}
  By inspecting the proof of  Theorem 4.7.1 in~\cite{Ol74},
  uniformity with respect to~$u$ is easy to verify. From Stirling's formula,
  we have
  \[
    \frac{\Gamma(1+a-b)\Gamma(b)}{\Gamma(a)} \sim \Gamma(b)a^{1-b}.
  \]
  Combination of this with~\eqref{eq:Te repr}
  and~\eqref{eq:int olver} yields the assertion.
\end{proof}

\section{Monotonicity of $|M(a,b,x)|$ with respect to~$\mathrm{Im}(a)$}\label{se:mon}

Let
\[
  f(t,x):=|M(a+it,b,x)|^2=M(a+it,b,x)M(a-it,b,x),
\]
which is an entire function of $t$ and $x$
with  a power series expansion
\begin{equation*}
f(t,x)=\sum_{m\geq0}\sum_{n\geq0}v_{mn}\frac{t^m}{m!}\frac{x^n}{n!},
\end{equation*}
where $v_{mn}=f^{(m,n)}(0,0)$. 
Using the power series of $M$ and the Cauchy product we obtain
\begin{equation*}
f(t,x)=\sum_{n\geq0}
\left(\sum_{k=0}^n\binom{n}{k}\frac{(a+it)_k}{(b)_k}\frac{(a-it)_{n-k}}{(b)_{n-k}}\right)\frac{x^n}{n!}.
\end{equation*}
Since $f(t,x)$ is an even function of $t$,
 $v_{mn}=0$ when $m>2\lfloor n/2\rfloor$ or $m\equiv1\mod2$. 

\begin{theorem}\label{thm:mon}
Suppose $a\geq b>0$. Then $v_{mn}\geq0$ for all $m\geq0$, $n\geq0$.
\end{theorem}
We give the proof of the theorem after some lemmas at the end of this section.
It immediately implies the following corollary, which is the main result of
the section.
\begin{corollary}\label{cor:mon}
Suppose $a\geq b>0$ and $x>0$. Then 
\[
  t\in\mathbb R_+\mapsto|M(a+it,b,x)|
\]
 is increasing.
\end{corollary}
\begin{lemma}
The function~$f$ solves the differential equation
\begin{equation}
-4t^2xf(t,x)+\sum_{k=0}^4p_k(x)f^{(0,k)}(t,x)=0,
\end{equation}
where
\begin{align*}
p_0(x)&=-2a(1-3b+2b^2)-2a(1-4b)x-4ax^2\\
p_1(x)&=b-3b^2+2b^3+(2a+b-8ab-6b^2)x+(2+8a+6b)x^2-2x^3,\\ 
p_2(x)&=(5b^2-b)x+(-3-4a-10b)x^2+5x^3,\\
p_3(x)&=(1+4b)x^2-4x^3,\\
p_4(x)&=x^3.
\end{align*}
\end{lemma}
\begin{proof}
Note that both $M(a+it,b,x)$ and $M(a-it,b,x)$ satisfy  second-order differential equations with polynomial coefficients,
namely the corresponding confluent hypergeometric differential equations.
Thus $f(t,x)$ also satisfies an ODE (with respect to~$x$) with polynomial coefficients.
In the combinatorial and symbolic computation literature, such functions
are called holonomic, or $D$-finite~\cite{Ze90b}.
The ODE for~$f$ can be computed with Mathematica by the command
 \verb|DifferentialRootReduce|.
\end{proof}
Some computations in the following proofs are not given in detail,
because they can be easily done with a computer algebra system.
For ease of notation we allow negative indices
and set $v_{mn}=0$ for $m<0$ or $n<0$.
\begin{lemma}
The power series coefficients of~$f$ satisfy the recursion
\begin{equation}\label{RecA}
A_{-1,n}v_{m,n+1}+A_{0,n}v_{m,n}+A_{1,n}v_{m,n-1}+A_{2,n}v_{m,n-2}=4nm(m-1)v_{m-2,n-1}
\end{equation}
with
\begin{align*}
A_{-1,n}&=b-3b^2+2b^3+(1-5b+5b^2)n+(-2+4b)n^2+n^3\\
A_{0,n}&=6ab-4ab^2-2a+(6a+11b-5-8ab-6b^2)n\\
&\quad+(9-4a-10b)n^2-4n^3\\
A_{1,n}&=(8-10a-6b+8ab)n+(-13+8a+6b)n^2+5n^3\\
A_{2,n}&=(-4+4a)n+(6-4a)n^2-2n^3.
\end{align*}
Use of our negative index convention shows that the recursion holds for $m\geq0$, $n\geq0$.
\end{lemma}
\begin{proof}
We extract coefficients from the differential equation by
\begin{equation*}
\left[\frac{t^m}{m!}\frac{x^n}{n!}\right]t^\ell x^k f^{(j)}(t,x)=(m-\ell+1)_\ell\ (n-k+1)_k\ v_{m-\ell,n-k+j}
\end{equation*}
and with our convention for negative indices this equation is true for all 
$m\geq0$, $n\geq0$, $\ell\geq0$, $k\geq0$, $j\geq0$.
Then we collect terms.
\end{proof}

Let us introduce the differences
\begin{equation*}
v'_{m,n}=v_{m,n}-v_{m,n-1},\quad
v''_{m,n}=v_{m,n}'-v_{m,n-1}',\quad
v'''_{m,n}=v_{m,n}''-v_{m,n-1}'',
\end{equation*}
and
\begin{equation}\label{u}
u_{m,n}'=4nm(m-1)v_{m-2,n-1}-4(n-1)m(m-1)v_{m-2,n-2}.
\end{equation}
Conversely
\begin{equation}\label{partial}
v_{m,n}''=v_{m,n-1}''+v'''_{m,n},
v_{m,n}'=v_{m,n-1}'+v''_{m,n},\quad
v_{m,n}=v'_{m,n}+v_{m,n-1}.
\end{equation}
\begin{lemma}
The differences satisfy the recursion
\begin{equation}\label{RecG}
G_{-1,n}v_{m,n+1}'''=G_{0,n}v_{m,n}'''+G_{1,n}v_{m,n-1}''+G_{2,n}v_{m,n-2}'+G_{3,n}v_{m,n-3}+u_{m,n}'
\end{equation}
with
\begin{align*}
G_{-1,n}&=b-3b^2+2b^3+(1-5b+5b^2)n+(-2+4b)n^2+n^3\\
G_{0,n}&=2a+7b-4-6ab+b^2+4ab^2-4b^3+(10-6a-9b+8ab-4b^2)n\\
&\quad +(4a+2b-8)n^2+2n^3\\
G_{1,n}&=6-6a+3b-4ab+8ab^2-6b^3+(6a-5b+8ab-3b^2-10)n\\
&\quad+(4+2b)n^2\\
G_{2,n}&=-2-4b+2ab+4ab^2-2b^3+(2+4b+b^2)n\\
G_{3,n}&=b^2.
\end{align*}
With our negative index convention this recursion is valid for $m\geq0$, $n\geq1$.
\end{lemma}
\begin{proof}
Take the difference of (\ref{RecA}) for $n$ and $n-1$ and rearrange terms. 
\end{proof}
\begin{lemma}
Suppose $a\geq b>0$ and $n\geq2$. Then 
\begin{equation}\label{positive}
G_{-1,n}\geq0,\quad
G_{0,n}\geq0,\quad
G_{1,n}\geq0,\quad
G_{2,n}\geq0,\quad
G_{3,n}\geq0.
\end{equation}
\end{lemma}
\begin{proof}
This follows from elementary analysis of the polynomials, or mechanically using
the Mathematica commands \verb|Simplify| and \verb|Reduce|; see also
\verb|CyclicDecoposition|.
\end{proof}
We can now prove the main results of this section
(Theorem~\ref{thm:mon} and its corollary).
\begin{proof}[Proof of Theorem~\ref{thm:mon}]
To show $v_{m,n}\geq0$ for all $m\geq0$, $n\geq0$ when $a\geq b>0$ we perform a nested induction. 
The outer induction is with respect to $m\geq0$, and the inner one with respect to $n\geq0$. There is a little
difficulty involved concerning
\begin{equation*}
v_{0,1}'''=\frac{2a}{b}-3,
\end{equation*}
which is the only term in the induction that can be negative. For ease of notation let
\begin{equation*}
\tilde v_{m,n}'''=\begin{cases}
0&m=0,n=1\\
v_{m,n}'''&\text{otherwise}.
\end{cases}
\end{equation*}
Let us define the statement
\begin{equation*}
\mathcal B(m,n)\ \equiv\ \big( 
\tilde v_{m,k}'''\geq0,\
v_{m,k}''\geq0,\
v_{m,k}'\geq0,\
v_{m,k}\geq0,\
\forall k\leq n\big).
\end{equation*}
Recall that we work under the assumption $a\geq b>0$.
If $m<0$ or $n<0$ or $m\equiv1\pmod 2,$ then $\mathcal B(m,n)$ is trivially true
due to our negative index convention.

Step $m=0$: For $n<0$
\begin{equation*}
\tilde v_{0,n}'''=0,\quad
v_{0,n}''=0,\quad
v_{0,n}'=0,\quad
v_{0,n}=0,
\end{equation*}
trivially. For $n=0$ we have
\begin{equation*}
\tilde v_{0,0}'''=1,\quad
v_{0,0}''=1,\quad
v_{0,0}'=1,\quad
v_{0,0}=1.
\end{equation*}
For $n=1$ we have
\begin{equation*}
\tilde v_{0,1}'''=0,\quad
v_{0,1}''=\frac{2a}{b}-2\geq0,\quad
v_{0,1}'=\frac{2a}{b}-1\geq0,\quad
v_{0,1}=\frac{2a}{b}\geq0.
\end{equation*}
For $n=2$ we have
\begin{equation*}
\tilde v_{0,2}'''=\frac{2 a (2 a b+a+b)}{b^2 (b+1)}-\frac{6 a}{b}+3\geq0,\quad
v_{0,2}''=\frac{2 a (2 a b+a+b)}{b^2 (b+1)}-\frac{4 a}{b}+1\geq0,
\end{equation*}
and
\begin{equation*}
v_{0,2}'=\frac{2 a (2 a b+a+b)}{b^2 (b+1)}-\frac{2 a}{b}\geq0,\quad
v_{0,2}=\frac{2 a (2 a b+a+b)}{b^2 (b+1)}\geq0.
\end{equation*}
So far we have shown $\mathcal B(0,n)$ for $n\leq 2$. 
%The start was a bit pedestrian, but 
Now the recursion can be applied. 
Note that $u_{0,n}'=0$ for all $n\in\mathbb Z$.
Suppose the hypothesis $\mathcal B(0,n)$ is true. Then we can show $\mathcal B(0,n+1)$
by the recursion~(\ref{RecG}), property (\ref{positive}) and (\ref{partial}).

We are at the basis of the outer induction, namely we have shown $\mathcal B(m,n)$
for all $m\leq0$ and $n\in\mathbb Z$. Recall that all coefficients
are zero when $m\equiv1\pmod2$ and we must show it for even $m\geq2$.

Suppose the outer induction hypothesis $\mathcal B(m-2,n)$ holds for some $m\geq2$ 
and all $n\in\mathbb Z$.
Note that $\mathcal B(m,n)$ holds trivially for $n<0$. Now we have
\begin{equation*}
v_{m,n}'''=0,\quad
v_{m,n-1}''=0,\quad
v_{m,n-2}'=0,\quad
v_{m,n-3}=0,\quad
n=0,\ldots,m-1.
\end{equation*}
So we can use $\mathcal B(m,n)$ for $n\leq m-1$ as inner induction basis.

Suppose the inner induction hypothesis $\mathcal B(m,n)$ is true. Then
inspection of~\eqref{u} shows that
\begin{equation*}
v_{m-2,n-1}'\geq0 \Rightarrow v_{m-2,n-1}\geq v_{m-2,n-2} \Rightarrow
u_{m,n}'\geq0.
\end{equation*}
Thus all terms in (\ref{RecG}) are non-negative and $\mathcal B(m,n+1)$ is true.
This concludes the induction.
\end{proof}
\begin{remark}
We conjecture that the conclusions of Theorem~\ref{thm:mon} and
Corollary~\ref{cor:mon} are true
for all $a\geq0$, $b>0$ and $x>0$. Some partial results can be obtained
by similar methods as above for the case $0\leq a<b$, but the  statements
we could obtain so far are involved and unsatisfactory.
\end{remark}

\section{The $a$-zeros are simple and negative}\label{se:zeros}

For $b,x>0,$ the $a$-zeros of $M(a,b,x)$ are simple and located
on the negative real line.
This follows from applying Sturm-Liouville theory
to the ODE~\eqref{eq:ode}; see Propositions~1 and~2
in~\cite{Li04} and the references given there.
We now give an alternative proof of this fact, which is inspired by 
a similar proof concerning the Bessel function~$J_{\nu}$
(see p.~482 in~\cite{Wa95}).
\begin{proposition}\label{prop:zeros}
  Let $b,x>0.$ Then all $a$-zeros of $M(a,b,x)$ are negative real and simple.
\end{proposition}
\begin{proof}
   First, observe that~\eqref{eq:M series} is an increasing function
of~$a\geq0,$ and since $M(0,b,x)=1,$ we see that there are no
$a$-zeros in $[0,\infty).$
  The function $y(x):=x^{b/2}e^{-x/2} M(a,b,x)$ satisfies the differential equation
  \[
    y''+P y=0,\quad P:=-\frac14+\frac{\tfrac12b-a}{x}+
    \frac{2b-b^2}{x^2},
  \]
  and  $\eta(x):=x^{b/2}e^{-x/2} M(\bar{a},b,x),$ where $\bar a$ denotes the
  complex conjugate, satisfies
  \[
    \eta''+Q \eta=0,\quad Q:=-\frac14+\frac{\tfrac12b-\bar a}{x}+
    \frac{2b-b^2}{x^2}.
  \]
  Then (see p.~133 in~\cite{Wa95})
  \[
     \int^x(P-Q)y\eta\,  dx = y \frac{d\eta}{dx}-\eta \frac{dy}{dx},
  \]
  and so, with $E:=x^{b/2}e^{-x/2},$
  \begin{align*}
    (\bar a -a)&\int^x t^{b-1}e^{-t}M(a,b,t)M(\bar a,b,t)dt\\
    &= E M(a,b,x)\big( E M(\bar a,b,x)\big)'
    - E M(\bar a,b,x)\big( E M( a,b,x)\big)' \\
    &= E^2\Big(
      M(a,b,x) \frac{d}{dx}M(\bar a,b,x)-M(\bar a,b,x)\frac{d}{dx}M( a,b,x)
    \Big).
  \end{align*}
  Therefore, for $a\in\mathbb{C}\setminus\mathbb{R},$
  \begin{multline}\label{eq:int M M}
    \int_0^x  t^{b-1}e^{-t}M(a,b,t)M(\bar a,b,t)dt \\
    =\frac{x^b e^{-x}}{\bar a-a}\Big(
      M(a,b,x) \frac{d}{dx}M(\bar a,b,x)-M(\bar a,b,x)\frac{d}{dx}M( a,b,x)
    \Big).
  \end{multline}
  Now let~$a$ be a zero of $M(a,b,x)$. If~$a$ is non-real, then the right-hand
  side of~\eqref{eq:int M M} vanishes, since $a\neq \bar a$ and $\bar a$ is a zero as well,
  but the left-hand side becomes
  \[
    \int_0^x  t^{b-1}e^{-t}|M(a,b,t)|^2dt >0,
  \]
  a contradiction.
  
  We now show that the $a$-zeros are simple.
 Analogously to~\eqref{eq:int M M},
  we have
   \begin{multline*}
    \int_0^x  t^{b-1}e^{-t}M(a,b,t)M(a',b,t)dt \\
    =\frac{x^b e^{-x}}{a'-a}\Big(
      M(a,b,x) \frac{d}{dx}M(a',b,x)-M(a',b,x)\frac{d}{dx}M( a,b,x)
    \Big)
  \end{multline*}
  for $a\neq a'$. Let $a<0$ be a zero, and $a'=a+h$ with $h\to 0.$ Then
  \begin{align*}
      \int_0^x  t^{b-1}e^{-t}M(a,b,t)^2dt &=
      -\lim_{h\to0} \frac{x^b e^{-x}}{h} M(a+h,b,x)\, \frac{d}{dx} M(a,b,x) \\
      &= -x^b e^{-x} \frac{d}{da} M(a,b,x)\, \frac{d}{dx} M(a,b,x),
  \end{align*}
  which shows that $(d/da) M(a,b,x)$ does not vanish.
\end{proof}

\bibliographystyle{siam}
\bibliography{literature}

\end{document}